\documentclass[reqno]{amsart}     
\usepackage{amsmath,amssymb,amsaddr}    
\usepackage{enumerate}
\usepackage{bbm}

\newtheorem{theorem}                   {Theorem} 
\newtheorem{thm}             [theorem] {Theorem} 
\newtheorem{lemma}           [theorem] {Lemma}

\newtheorem{cor}             [theorem] {Corollary}   
\newtheorem{proposition}     [theorem] {Proposition}

\theoremstyle{remark}

\newcommand{\eps}{\varepsilon}
\newcommand{\cH}{\mathcal{H}}

\newcommand{\cI}{\mathcal{I}}
\newcommand{\EE}{\mathbb{E}}

\newcommand{\Bi}{\mathrm{Bi}}

\newcommand{\Ind}[1]{\mathbbm{1}_{[#1]}}

\newcommand{\cond}{\; \middle\vert \;}

\newcommand{\conntime}{{\tau_c}}
\newcommand{\isoltime}{{\tau_i}}
\newcommand{\process}{{\{\cH^k(n,M)\}_M}}
\newcommand{\hknm}{{\cH^k(n,M)}}
\newcommand{\hknp}{{\cH^k(n,p)}}
\newcommand{\omone}{{\omega_1}}
\newcommand{\Po}{\mathrm{Po}}
\renewcommand{\Pr}{\mathbb{P}}

\begin{document}
\title[High-order connectivity in random hypergraphs]{Threshold and hitting time for high-order connectivity in random hypergraphs}
\thanks{The authors are supported by Austrian Science Fund (FWF): P26826\\
The second author is additionally supported by Austrian Science Fund (FWF): W1230, Doctoral Program ``Discrete Mathematics''.}

\author[O.~Cooley, M.~Kang and C.~Koch]{Oliver Cooley, Mihyun Kang and Christoph Koch}
\email{\{cooley,kang,ckoch\}@math.tugraz.at}
\address{Institute of Optimization and Discrete Mathematics,\\ Graz University of Technology, 8010 Graz, Austria}

\date{\today}

\begin{abstract}
We consider the following definition of connectivity in $k$-uniform hypergraphs: Two $j$-sets are $j$-connected if there is a walk of edges between them such that two consecutive edges intersect in at least $j$ vertices. We determine the threshold at which the random $k$-uniform hypergraph with edge probability $p$ becomes $j$-connected with high probability. We also deduce a hitting time result for the random hypergraph process -- the hypergraph becomes $j$-connected at exactly the moment when the last isolated $j$-set disappears. This generalises well-known results for graphs.
\end{abstract}

\maketitle
\noindent Keywords: \emph{random hypergraphs, connectivity, hitting time}\\
Mathematics Subject Classification: 05C65, 05C80

\section{Introduction}

\subsection{Preliminaries and main results}
In the study of random graphs, one very famous result concerns the \emph{hitting time} for connectivity. More precisely, if we add randomly chosen edges one by one to an initially empty graph on $n$ vertices, then with high probability at the moment the last isolated vertex gains its first edge, the whole graph will also become connected (this classical result was first proved by Bollob\'as and Thomason in~\cite{BT85}). This interplay between local and global properties is an example of the common phenomenon relating graph properties with their smallest obstruction: The graph can certainly not be connected while an isolated vertex still exists, but this smallest obstruction is also the critical one which is last to disappear.

In this paper we generalise this result to random $k$-uniform hypergraphs. For an integer $k\ge 2$, a $k$-uniform hypergraph consists of a set $V$ of vertices together with a set $E$ of edges, each consisting of $k$ vertices. (The case $k=2$ corresponds to a graph.) We need to define the notion of connectivity, for which there is a whole family of possible definitions. For any $1\le j \le k-1$, we say that two $j$-sets (of vertices) $J_1,J_2$ are \emph{$j$-connected} if there is a sequence of edges $E_1,\ldots,E_m$ such that
\begin{itemize}
\item $J_1 \subseteq E_1$ and $J_2\subseteq E_m$;
\item $|E_{i}\cap E_{i+1}|\ge j$ for all $1\le i \le m-1$.
\end{itemize}
In other words, we may walk from $J_1$ to $J_2$ using edges which consecutively intersect in at least $j$ vertices. A \emph{$j$-component} is a maximal set of pairwise $j$-connected $j$-sets.

Note that in the case $k=2,j=1$ this is simply the usual definition of connectedness for graphs. More generally, for arbitrary $k\ge 2$ the case $j=1$ is by far the most well-studied. This is not necessarily because the definition is more natural, but rather because it is much easier to visualise and the analysis is often significantly simpler. In this paper we will be interested in arbitrary $1\le j \le k-1$ and $k\ge 3$.

There is also more than one model for random hypergraphs. We first define the \emph{uniform model}: Given any natural numbers $k,M,n$ such that $M\le \binom{n}{k}$, the random hypergraph $\hknm$ is a hypergraph chosen uniformly at random from all hypergraphs on vertex set $\{1,\ldots,n\}$ which have $M$ edges. This is closely related to the \emph{random hypergraph process} $\process$ which is defined as follows:
\begin{itemize}
\item $\cH^k(n,0)$ is the graph on vertex set $\{1,\ldots,n\}$ with no edges;
\item For $1\le M \le \binom{n}{k}$, $\hknm$ is obtained from $\cH^k(n,M-1)$ by adding an edge chosen uniformly at random from among those not already present.
\end{itemize}
Note that the distribution of the random hypergraph obtained in the $M$-th step of the process is the same as in the uniform model $\hknm$, so the notation is consistent.

We consider asymptotic properties of random hypergraphs and throughout this paper any asymptotics are as $n\to \infty$. In particular we say \emph{with high probability} (or \emph{whp}) to mean with probability tending to $1$ as $n\to \infty$.

We say that a $k$-uniform hypergraph is \emph{$j$-connected} if there is one component which contains all $j$-sets. A $j$-set is \emph{isolated} if it is not contained in any edges. It is trivial that if a hypergraph contains isolated $j$-sets, then it is not $j$-connected (assuming it has more than $j$ vertices). Our main result is that this trivial smallest obstruction is also the critical one in a random hypergraph.

Let $\conntime=\conntime(n,j,k)$ denote the time step in the hypergraph process $\process$ at which the hypergraph becomes $j$-connected. Similarly, let $\isoltime$ denote the time at which the last isolated $j$-set disappears. Note that the properties of being $j$-connected or of having no isolated $j$-set are certainly monotone increasing properties, so these two variables are well-defined.

\begin{thm}\label{thm:hittingtime}
For any $1\le j \le k-1$ and $k\ge 3$, with high probability in the random hypergraph process $\process$ we have $\conntime = \isoltime$.
\end{thm}

The case $j=1$ of this theorem was already proved as a special case of the results in~\cite{Poole14}.

The uniform model and hypergraph process allow us to formulate exact hitting time results such as Theorem~\ref{thm:hittingtime}. However, the drawback is that the analysis of the model can become tricky due to the fact that the presence of different edges is not independent (the total number is fixed). For this reason, it is often easier to analyse the \emph{binomial model}: $\hknp$ is a random $k$-uniform hypergraph on vertex set $\{1,\ldots,n\}$ in which each $k$-set is an edge with probability $p$ independently. In Section~\ref{sec:contiguity} we will show that if $p=M/\binom{n}{k}$, then the two models are very similar and we can transfer results from one model to the other.

For the proof of Theorem~\ref{thm:hittingtime} we will also make use of the following result (Theorem~\ref{thm:poisson}), which is interesting in itself and is therefore stated in a significantly more general form than we need for Theorem~\ref{thm:hittingtime}. For integer valued random variables $Z$ and $Z'$ we denote their \emph{total variation distance} by $d_{TV}(Z,Z')$, i.e.\
\begin{equation*}
d_{TV}(Z,Z')=\frac{1}{2}\sum_{i}\left|\Pr\left(Z=i\right)-\Pr\left(Z'=i\right)\right|.
\end{equation*}
For integer-valued random variables $X_n$ and $Y$, we say \emph{$X_n$ converges in distribution to $Y$}, denoted by $X_n \stackrel{d}{\longrightarrow} Y$, if for every integer $i$ we have $\Pr(X_n=i)\rightarrow \Pr(Y_n=i)$.

\begin{theorem}\label{thm:poisson}
Suppose $p=\frac{j\log n+s\log\log n +c_n}{\binom{n}{k-j}}$, where $c_n=o(\log n)$. For any integer $s\ge0$ let $D_s$ be the number of $j$-sets of degree precisely $s$ in $\cH^k(n,p)$ (i.e.\ which lie in $s$ edges). Then we have
\begin{equation}\label{Eq:TVBound}
d_{TV}\left(D_s,\Po \left(\EE\left(D_s\right)\right)\right)=O(n^{-j}\log n).
\end{equation}
In particular for any constants $s,c$, we have
\begin{center}
\begin{tabular}{rll}
$(i)$ & $D_s=0$ whp & if $c_n\rightarrow \infty$;\\
$(ii)$ & $D_s \stackrel{d}{\longrightarrow} \Po\left(\frac{j^s e^{-c}}{j!s!}\right)$ & if $c_n\rightarrow c$;\\
$(iii)$ & $D_s\to \infty$ whp & if $c_n \to -\infty$.
\end{tabular}\end{center}
\end{theorem}

These two theorems together give the following immediate corollary.

\begin{cor}
Let $p=\frac{j\log n + c_n}{\binom{n}{k-j}}$.
\begin{enumerate}
\item If $c_n\rightarrow -\infty$ then with high probability $\hknp$ contains isolated $j$-sets (and is therefore not $j$-connected).
\item If $c_n\rightarrow \infty$ then with high probability $\hknp$ is $j$-connected (and therefore contains no isolated $j$-sets).
\end{enumerate}
In other words, the properties of being $j$-connected and having no isolated $j$-sets both undergo a (sharp) phase transition at threshold
\[
p_0 = \frac{j\log n}{\binom{n}{k-j}}.
\]
\end{cor}

\subsection{Methods}
The main contribution of this paper is to deduce Theorem~\ref{thm:hittingtime} from Theorem~\ref{thm:poisson}. Attempting to prove this directly using standard techniques generalised from the graph case does not work because $j$-components in a hypergraph may be strangely and non-intuitively distributed. To overcome this problem we quote a powerful result from~\cite{CoKaKo14}, which guarantees one component with a large subset which is in some sense \emph{smoothly distributed}. We then show that with high probability all non-trivial components are connected to this smooth subset.

\subsection{Notation and definitions}
We introduce a few more definitions before we proceed with the proofs. We fix $k\ge 3$ and $1\le j\le k-1$ for the remainder of the paper. The \emph{order} of a hypergraph is the number of vertices it contains, while its \emph{size} is the number of edges. Since a $j$-component consists of $j$-sets of vertices, we sometimes view it as a $j$-uniform hypergraph in which the edges are the $j$-sets in the component. (In particular, the size of a $j$-component is the number of $j$-sets it contains.)

We will sometimes need to relate the $j$-sets of a component to the edges of the hypergraph which connect them. To allow us to do this, for a $k$-uniform hypergraph $H$ we define the \emph{$j$-size} of $H$ to be the number of $j$-sets contained in edges of $H$.

\section{Contiguity of $\hknm$ and $\hknp$}\label{sec:contiguity}
We need to know that $\hknp$ and $\hknm$ are roughly equivalent. We quote a result from~\cite{JLRBook}, which in turn is based on previous arguments by Bollob\'as and {\L}uczak. In fact, \cite{JLRBook} considers a more general setting than we require here, but what we state is an immediate corollary of the results there (see~\cite{JLRBook}, Corollary~1.16). Let $N=\binom{n}{k}$ and to ease notation, for some property $Q$ we will denote by $\Pr_M(Q)= \Pr(\hknm\in Q)$ the probability that $\hknm$ has property $Q$. $\Pr_p(Q)$ is defined similarly.

\begin{lemma}
Let $Q$ be some monotone increasing property of $k$-uniform hypergraphs and let $M=Np\rightarrow \infty$. Then
\begin{enumerate}\label{lem:monotonecontiguity}
\item $\Pr_p(Q)\rightarrow 1$ implies $\Pr_M(Q)\rightarrow 1$;
\item $\Pr_p(Q)\rightarrow 0$ implies $\Pr_M(Q)\rightarrow 0$.
\end{enumerate}
\end{lemma}
This lemma allows us to transfer properties from $\hknp$ to $\hknm$ (transferring in the other direction is also possible, with some small modifications, but we will not need to do this here). However, this only works for monotonically increasing properties. This is fine for the properties of being $j$-connected or of having no isolated $j$-sets. However, in the proof of Theorem~\ref{thm:hittingtime} we will need to consider the probability of having a component of size $r$, for various fixed $r$. This property is not even convex (and nor is its complement) and so for this case we will need some more careful arguments.

The following standard argument allows us to transfer high probability events from the binomial to the uniform model provided that the failure probability is small enough.
\begin{lemma}\label{lem:ptoM}
Let $Q$ be an arbitrary property of $k$-uniform hypergraphs. Suppose $M\rightarrow \infty$ and $p=M/N \rightarrow 0$. Then
\[
\Pr_M(Q) \le \frac{\Pr_p(Q)}{\Pr(e(\cH^k(n,p)=M)} = \Theta(M^{1/2})\Pr_p(Q).
\]
\end{lemma}

\begin{proof}
The inequality follows from the fact that
\begin{align*}
\Pr_p(Q) & = \sum_{m=0}^N \Pr_m(Q)\Pr(e(\cH^k(n,p))=m)\\
& \ge \Pr_M(Q)\Pr(e(\cH^k(n,p))=M).
\end{align*}
For the equality we use Stirling's approximation to deduce that
\begin{align*}
\Pr(e(\cH^k(n,p))=M) & = \binom{N}{M} p^M (1-p)^{N-M}\\
& =  \Theta (1)\sqrt{\frac{N}{M(N-M)}}\frac{N^N}{M^M (N-M)^{N-M}}p^M (1-p)^{N-M}\\
& = \Theta(M^{-1/2}).\qedhere
\end{align*}
\end{proof}

\section{Proof of Theorem~\ref{thm:poisson}}\label{sec:poisson}

Let $C=\binom{k}{j}-1$. Fix an integer $s\ge 0$ and suppose $p=\frac{j\log n+s\log\log n +c_n}{\binom{n}{k-j}}$, where $c_n=o(\log n).$ Then the expected number of $j$-sets of degree $s$ satisfies
\begin{align}
\nonumber\EE(D_s)&=\binom{n}{j}\binom{\binom{n-j}{k-j}}{s}p^s(1-p)^{\binom{n-j}{k-j}-s}\\
\nonumber&=(1+o(1))\frac{n^j}{j!}\frac{\left(\frac{n^{k-j}}{(k-j)!}\right)^s}{s!}p^s \exp\left(-p\binom{n}{k-j}\right)\\
\nonumber&=(1+o(1))\frac{1}{j!s!}e^{s(k-j)\log n- s\log((k-j)!)+s\log p -s\log\log n -c_n}\\
&=(1+o(1))\frac{j^s}{j!s!}e^{-c_n},\label{Eq:Expectation}
\end{align}
since
\begin{align*}
\log p=-(k-j)\log n+\log\log n+\log(j(k-j)!) +O\left(\frac{\log\log n+\left|c_n\right|}{\log n}+\frac{1}{n}\right).
\end{align*}

For the Poisson-approximation we use the Chen-Stein method (cf.~\cite{BarbourHolstJanson92}). For any $j$-set $J$ we denote its degree in $\cH^k(n,p)$ by $\deg(J)$ and analyse how $D_s$ changes by conditioning on the event $\{\deg(J_0)=s\}$ for an arbitrary  $j$-set $J_0$. 

 First we construct $\cH^k(n,p)$ and denote by $E_0$ the set of edges containing $J_0$, then we distinguish three cases:
\begin{enumerate}[(a)]
\item If $\deg(J_0)<s$, add $s-\deg(J_0)$ distinct $k$-sets chosen uniformly at random from $\left\{K\in \binom{V}{k}\cond J_0\subset K\right\}\setminus E_0$ to the hypergraph;
\item If $\deg(J_0)=s$, do nothing;
\item If $\deg(J_0)>s$, delete a set of $\deg(J_0)-s$ edges chosen uniformly at random from $E_0$.
\end{enumerate}
We denote the resulting hypergraph by $\cH^*=\cH^*(J_0)$. For any $j$-set $J$ we write $\deg^*(J)$ for its degree in $\cH^*$ and $D_s^* = D_s^*(J_0)$ for the number of $j$-sets $J\neq J_0$ such that $\deg^*(J)=s$. Furthermore observe that this construction provides a coupling of $\cH^k(n,p)$ and $\cH^*$ such that removing all edges incident with $J_0$ in either one of them yields the same random hypergraph $\cH^-=\cH^-(J_0)$.  For any $j$-set  $J$ we write $\deg^-(J)$ for its degree in $\cH^-$. 

We use the following form of the Chen-Stein approximation given by Theorem~1.B in~\cite{BarbourHolstJanson92}.
\begin{theorem}[Chen-Stein approximation \cite{BarbourHolstJanson92}]\label{thm:chen-stein}
Given a finite index set $\cI$ and a random variable  $W=\sum_{i\in \cI} Z_i$, where $Z_i$ is a Bernoulli random variable with parameter $p_i\in [0,1]$ and denote by $\lambda=\sum_{i\in \cI} p_i$ its expectation. Assume that for each $i\in \cI$ there is a pair of coupled random variables $(U_i,V_i)$ such that $U_i$ has the distribution of $W$ and $V_i+1$ has the distribution of $W$ conditioned on $Z_i=1$. Then we have 
\begin{align*}
d_{TV}\big(W,\Po(\lambda)\big)&\leq \min\{1,\lambda^{-1}\}\sum_{i\in \cI}p_i\EE\left(\left|U_i-V_i\right|\right).
\end{align*}
 \end{theorem}
For the proof of Theorem~\ref{thm:poisson}, we let $\cI$ be the set of all $j$-sets and for all $J$ let $Z_J=\Ind{\deg(J)=S}$, $p_{J}=\Pr\left(\deg(J)=s\right)$, $U_{J}=W=D_s$ and $V_{J}=D_s^*(J_0)$ we obtain
\begin{align}
d_{TV}\big(D_s,\Po(\EE(D_s))\big)&\leq \frac{\sum_{J}\Pr\left(\deg(J)=s\right)\EE\left(\left|D_s-D_s^*(J_0)\right|\right)}{\EE\left(D_s\right)}=\EE\left(\left|D_s-D_s^*\right|\right),\label{Eq:ChenSteinBound}
\end{align}
since $D_s^*(J_0)$ has the same distribution for all $J_0$ by symmetry. Hence it suffices to estimate the random variable $\left|D_s-D_s^*\right|.$

Note that $J_0$ contributes to $\left|D_s-D_s^*\right|$ only if $\deg(J_0)=s$ and in that case no other $j$-set contributes.
If $\deg(J_0)<s$, say $\deg(J_0)=s-t$ for some $t\in[1,s]$, then the only contribution to $\left|D_s-D_s^*\right|$ comes from
$j$-sets $J\neq J_0$ whose degree increased, i.e.\ $\deg^*(J)>\deg(J)$. Moreover there are at most $Ct$ such $j$-sets and it will be
a sufficiently good upper bound to estimate their contribution to $\left|D_s-D_s^*\right|$ by $1$, even though some may not actually contribute.
Similarly, if $\deg(J_0)=s+t$ for some $t\in\big[1,\binom{n-j}{k-j}-s\big]$, then there can also be at most $Ct$ $j$-sets that
could potentially contribute. However, we have to be more careful and observe that for a $j$-set $J$ to contribute it is necessary
to have either $\deg(J)=s$ or $\deg^*(J)=s$. Note that these cannot hold unless $\deg^-(J)\le s$, and we will simply bound the probability of this (more likely) event. Note that $\deg^-(J)$ has distribution
\[
\Bi \left(\binom{n-j}{k-j}-\binom{n-|J_0\cup J|}{k-|J_0\cup J|},p\right),
\]
and the probability that $\deg^-(J)\le s$ is maximised when $|J_0\cup J|$ is minimised. Hence for an upper bound we will assume that $|J_0\cup J|=j+1$, and
by symmetry we may again fix an arbitrary $j$-set $J_1$ satisfying $|J_0\cup J_1|=j+1$. Combining all these arguments we obtain the upper bound
\begin{align*}
\left|D_s-D_s^*\right|\leq\Ind{\deg(J_0)=s} &+\sum_{t=1}^{s}\Ind{\deg(J_0)=s-t}\, Ct+ \hspace{-0.2cm} \sum_{t=1}^{\binom{n-j}{k-j}-s} \hspace{-0.2cm} \Ind{\deg(J_0)=s+t}\Ind{\deg^-(J_1)\le s}\, Ct. 
\end{align*}
Therefore, using the notation $x^+:= \max\{x,0\}$ for any $x\in \mathbb{R}$, we have
\begin{align}\label{Eq:TVError}
\EE\left(\left|D_s-D_s^*\right|\right)\leq \Pr\left(\deg(J_0)=s\right)+C\EE\left(s-\deg(J_0)\right)^+ +C\EE\left(\deg(J_0)-s\right)^+ q,
\end{align}
where 
\begin{align*}
q&=\Pr\left(\deg^-(J_1)\le s\cond \deg(J_0)=s+t \right)=\Pr\left(\deg^-(J_1)\le s\right)
\end{align*}
since $\cH^-$ is independent from the set of edge indicators corresponding to $k$-sets containing $J_0$\,. Both probabilities in \eqref{Eq:TVError} are bounded from above by
\begin{align*}
\Pr\left(\Bi\left(\binom{n-j}{k-j}-\binom{n-j-1}{k-j-1}, p\right)\le s\right)\le \exp\left(-\binom{n}{k-j}p/3\right)=O(n^{-j})
\end{align*}
by a Chernoff bound, since $s$ is bounded. Moreover we have 
\begin{align*}
\EE\left(s-\deg(J_0)\right)^+\leq s\,\Pr\left(\deg(J_0)\le s\right)=O(n^{-j})
\end{align*}
and 
\begin{align*}
\EE\left(\deg(J_0)-s\right)^+\leq \EE(\deg(J_0))+s=O(\log n).
\end{align*}
Therefore~\eqref{Eq:ChenSteinBound} and~\eqref{Eq:TVError} provide~\eqref{Eq:TVBound}, i.e.\ 
\begin{equation}
d_{TV}\left(D_s,\Po\left(\EE\left(D_s\right)\right)\right)=O(n^{-j}\log n).
\end{equation}
Now assume $\lim_{n\to\infty}c_n= c$. By \eqref{Eq:Expectation} we know that $\EE\left(D_s\right)\to \frac{j^s e^{-c}}{j!s!}$ and by the continuity in $\lambda$ of the function $\Pr(\Po(\lambda)=i)$ for each $i$
$$
\Po\left(\EE\left(D_s\right)\right)\stackrel{d}{\longrightarrow}\Po\left(\frac{j^s e^{-c}}{j!s!}\right),
$$
hence by the triangle inequality and \eqref{Eq:TVBound}, case~$(ii)$ in the second claim follows. Cases~$(i)$ and~$(iii)$ can be easily deduced from case~$(ii)$.

\section{Proof of Theorem~\ref{thm:hittingtime}}\label{sec:hittingtime}

The proof which we present is largely elementary except for the use of Theorem~\ref{thm:poisson}, which relies on Theorem~\ref{thm:chen-stein} and one powerful result from~\cite{CoKaKo14}. This result is stated for a much smaller probability than we have in this setting, which is therefore not the optimal range for its application, but nevertheless it will turn out to be strong enough.

\begin{lemma}\label{lem:smoothsubset}
Suppose $n^{-1/3} \ll \eps \ll 1$ and
let $p^*=\frac{1+\eps}{\left(\binom{k}{j}-1\right)\binom{n}{k-j}}$. Then with high probability there is a $j$-component of $\cH^k(n,p^*)$ with a subset $S$ of $j$-sets satisfying the following property:
\begin{center}
\parbox{11cm}{
Every $(j-1)$-set of vertices in $\cH^k(n,p^*)$ is contained in $(1\pm o(1))\eps^3 n$ $j$-sets of $S$.
}
\end{center}
\end{lemma}
In other words, we can find a reasonably large subset $S$ of a component which is \emph{smooth} in the sense that all $(j-1)$-sets are in about the ``right'' number of $j$-sets of $S$.

We note that Lemma~\ref{lem:smoothsubset} is not stated explicitly in this form in~\cite{CoKaKo14}, but is implicit in the proof. More precisely, it is proved that with high probability there is a component of size $\Theta(\eps n^j)$ and that with high probability, starting from any $j$-set in this component, a breadth-first search process produces smooth generations from some starting generation $g_0$ up to some stopping generation $g_1$ (Lemma 16 in~\cite{CoKaKo14}). At generation $g_1$, either the boundary has size at least $\eps^3 n^j$, in which case we take this boundary as our set $S$, or the whole component so far has size $\eps^{3/2}n^j$. Furthermore, with high probability $g_0$ is small (Lemmas~24 and~25 in~\cite{CoKaKo14}), so the non-smooth portion of the component is negligible. In the second case, we therefore take the portion of the component between the starting and stopping times as our $S$ and because each generation is smooth, their union is also smooth.

We now proceed with the proof of Theorem~\ref{thm:hittingtime}. Let us consider any $p,M$ satisfying
\[
\frac{j\log n - \omega}{\binom{n}{k-j}}\le p =M/N \le \frac{j\log n + \omega}{\binom{n}{k-j}}
\]
where $\omega:=\log \log n$ and observe that by Theorem~\ref{thm:poisson} and Lemma~\ref{lem:monotonecontiguity}, in both $\cH^k(n,p)$ and $\cH^k(n,M)$, with high probability there are isolated $j$-sets at the lower end of this range but not at the upper end. We would now like to say that other than these isolated $j$-sets, there is just one very large component.

We set $p^\dagger:= \frac{p-p^*}{1-p^*}$ and set $\cH_1:= \cH(n,p^*)$ and $\cH_2:=\cH(n,p^\dagger)$. Observe that we may couple in such a way that $\cH^k(n,p) = \cH_1\cup \cH_2$. Furthermore, by Lemma~\ref{lem:smoothsubset}, with high probability $\cH_1$ has a component containing a smooth set $S$. In $\cH^k(n,p)$ this component may be bigger than in $\cH_1$, but certainly still contains $S$. We consider the possibility that there is a second non-trivial component containing $r$ $j$-sets, and make a case distinction on the size of $r$.

In all cases we will use the following proposition. We say that a hypergraph is $\emph{well-constructed}$ if can be generated from an initial $j$-set via a search process, i.e.\ by successively adding edges such that each edge contains at least one previously discovered $j$-set, and such that each edge also contains at least one previously undiscovered $j$-set. Note that for any $j$-component of size $r$ in a $k$-uniform hypergraph, there is a well-constructed subhypergraph of every (up to a constant $\binom{k}{j}$ error) $j$-size up to $r$.

\begin{proposition}\label{prop:structurecount}
Up to isomorphism, the number of well-constructed $k$-uniform hypergraphs of $j$-size $s$ is at most $2^{ks^2}$.
\end{proposition}

\begin{proof}
We explore the hypergraph by adding the edges one by one in the order in which it is well-constructed. The resulting hypergraph is uniquely determined, up to isomorphism, by the intersection of each edge with the previous vertices (though we will multiple count the isomorphism classes, this is permissible for an upper bound). When adding the $i$-th edge, we certainly have at most $(i-1)k$ vertices so far, and so the number of possible intersections is at most $2^{(i-1)k}$. Multiplying over all edges, of which there are certainly at most $s$ (each edge gives at least one new $j$-set), we have that the number of such hypergraphs is at most $2^{\sum_{i=1}^s(i-1)k} \le 2^{ks^2}$.
\end{proof}

We now continue with the Proof of Theorem~\ref{thm:hittingtime}. Let us set $\omone:=\log \log n$.

\smallskip

\noindent \textbf{Case 1:} $2\le r \le \omone$.

Let us first observe that in a component of size $r \ge 2$ we must have at least one edge, and therefore at least $\binom{k}{j} \ge k \ge 3$ $j$-sets, i.e. we automatically have $r\ge 3$.

We show that the expected number of components of size $r$ is very small and apply Markov's inequality. Any component of size $r$ can be associated with a well-constructed hypergraph $H$ of $j$-size $r$ which is isolated from the remaining $j$-sets of $\hknp$. Then $e(H)\le r$ and furthermore $|H|\le j+(k-j)e(H)$, since each new edge of $H$ gives at most $k-j$ new vertices. For each $j$-set of $H$, we have at least $\binom{n-j}{k-j}-r\binom{n-j-1}{k-j-1}$ non-edges (any $k$-set containing this $j$-set but no other $j$-sets of $H$). Thus the expected number of isolated copies of $H$ in $\cH^k(n,p)$ satisfies
\[
\EE (X_H) \le n^{j+(k-j)e(H)}p^{e(H)}(1-p)^{r\left(\binom{n-j}{k-j}-r\binom{n-j-1}{k-j-1}\right)}
\]
and so
\begin{align*}
\log(\EE (X_H)) \le \; & \left(j+(k-j)e(H)\right)\log n+O(r\log \log n) \\
& \hspace{0.5cm}- (k-j)e(H)\log n -(1-O(r/n)-O(\omega/\log n))rj\log n\\
= \; & (1-r+o(1))j\log n \le (-3rj/5)\log n.
\end{align*}
Note that this bound does not depend on the specific structure of $H$, only on the number of $j$-sets $r$. Let $X_r$ be the number of components of size $r$. Then by Proposition~\ref{prop:structurecount} we have
\[
\EE (X_r) \le 2^{kr^2} n^{-3rj/5} \le n^{-4rj/7}
\]
where for the last inequality we use the fact that $r\le \omone = o(\log n)$.

By taking a union bound over all $3\le r \le \omone$, we conclude that with probability at least $1-2n^{-12j/7}$ there are no $j$-components of this size.

\smallskip

\noindent \textbf{Case 2:} $r \ge \omone$.

In this case, rather than looking at the full component we look at a well-constructed subgraph $H$ of $j$-size $\omone$. Such a subgraph certainly exists up to a $\binom{k}{j}$ error term in the $j$-size, which will not affect calculations significantly. Many of the calculations from Case~1 are still valid, replacing $r$ by $\omone$. However, since we are no longer considering a full component, we must be more careful about the number of non-edges.

At this point we make use of the set $S$ of $j$-sets which lie in a different component to $H$. For each of the $\omone$ $j$-sets of $H$, pick an arbitrary $(j-1)$-set within it and by Lemma~\ref{lem:smoothsubset}, this $(j-1)$-set is contained in $(1\pm o(1))\eps^3 n$ $j$-sets of $S$. For each such pair of $j$-sets intersecting in $j-1$ vertices, there are $\binom{n-j-1}{k-j-1}$ $k$-sets containing both of them, all of which must be non-edges, since the $j$-sets lie in different components.

It may be that we multiple count the non-edges in this way. However, each $k$-set may only be counted from a pair of $j$-sets it contains, and therefore the number of times it is counted is certainly at most $\binom{k}{j}(k-j) \le 2^k$. Thus in total the number of non-edges is at least
\[
2^{-(k+1)}\omone \eps^3 n \binom{n}{k-j-1} = \Theta\left(\omone\eps^3 n^{k-j}\right).
\]
We may thus calculate the expected number of such structures $H$:
\[
\EE (X_H) \le n^{j+(k-j)e(H)}p^{e(H)}(1-p)^{\Theta(\omone \eps^3 n^{k-j})}
\]
and so, letting $Y$ be the number of such well-constructed hypergraphs of $j$-size $\log \log n$ which are not in the same component as $S$, we have
\begin{align*}
\log(\EE (Y)) & \le k\omone^2\log 2 + j\log n+O\left(\omone \log \log n \right) - \Theta\left(\omone \eps^3 \log n\right).
\end{align*}
Now observe that in Lemma~\ref{lem:smoothsubset} we may choose any $n^{-1/3} \ll \eps \ll 1$. In particular, choosing $\eps^3=\frac1{\log \log \log n}$, we have $\omone \eps^3 \to \infty$ and the last term in the above inequality dominates, and we have $\log(\EE(Y)) \le -C_0\log n$ for any constant $C_0$. In particular, choosing $C_0=12j/7$, we have $\EE(Y)\le n^{-12j/7}$. By Markov's inequality, this implies that with probability at least $1-n^{-12j/7}$ we have $Y=0$ and therefore no further components of size $r$.

Combining the two cases, this tells us that with probability at least $1-3n^{-12j/7}$, $\cH^k(n,p)$ only has one non-trivial component.

Finally note that $M=pN = \Theta(n^j\log n)$. Thus by Lemma~\ref{lem:ptoM} we conclude that with probability at least $1-3n^{-12j/7}\sqrt{M} = 1-o(n^{-8/7})$, $\cH^k(n,M)$ also has only one non-trivial $j$-component.

We now take a union bound over all possible $M$, of which there are at most $\frac{2\omega}{\binom{n}{k-j}}\binom{n}{k}=O(\omega n^j)$, and deduce that the probability that there is ever a second non-trivial $j$-component within this time period is at most
\[
O(\omega n^j)n^{-8j/7} = O(\omega n^{-j/7}) = o(1)
\]
as required.

\section{Concluding remark}
In~\cite{Poole14}, it is determined for the case $j=1$ that the hitting time for \emph{d-strong} $1$-connectedness, i.e.\ the time at which the hypergraph first has the property that deleting any set of less than $d$ vertices still leaves a $1$-connected hypergraph, is the same as the hitting time for having no vertices of degree less than $d$ with high probability. It would be interesting to generalise this result to $d$-strong $j$-connectedness (removing fewer than $d$ $j$-sets still leaves a $j$-connected hypergraph), which is presumably attained with high probability when every $j$-set has degree at least $d$. However, this would present significant additional difficulties, not least that Lemma~\ref{lem:smoothsubset} would no longer give the substructure which we require.

\ 

\bibliographystyle{amsplain}
\bibliography{CTBibliography}

\ 

\end{document}